\documentclass{article}
\usepackage{amsfonts,amsmath,amsthm,amssymb}
\usepackage{graphics, epsfig}
\usepackage{color}
\usepackage{appendix}
\usepackage{ulem}
\usepackage[makeroom]{cancel}
\usepackage{fancyhdr}
\usepackage{centernot}
\usepackage{mathtools}
\usepackage{ stmaryrd }

 \usepackage[usenames,dvipsnames]{pstricks}
 \usepackage{pst-grad} 
 \usepackage{pst-plot} 
\allowdisplaybreaks

\let\TeXchi\chi
\newbox\chibox
\setbox0 \hbox{\mathsurround0pt $\TeXchi$}
\setbox\chibox \hbox{\raise\dp0 \box 0 }
\def\chi{\copy\chibox}

\newtheorem{proposition}{Proposition}[section]
\newtheorem{theorem}{Theorem}[section]
\newtheorem{definition}{Definition}[section]

\newtheorem{remark}{Remark}[section]

\numberwithin{equation}{section}
\numberwithin{theorem}{section}
\numberwithin{definition}{section}
\numberwithin{example}{section}
\numberwithin{proposition}{section}
\numberwithin{lemma}{section}
\numberwithin{remark}{section}
\DeclareMathOperator{\Ima}{Im}

\DeclareMathOperator{\Hom}{Hom}
\DeclareMathOperator{\rk}{rk}
\newcommand\blfootnote[1]{%
  \begingroup
  \renewcommand\thefootnote{}\footnote{#1}%
  \addtocounter{footnote}{-1}%
  \endgroup
}
\begin{document}
\title{(Co)homology theories for structured spaces arising from their corresponding poset}
\author
{Manuel Norman}
\date{}
\maketitle
\begin{abstract}
\noindent In [1] we introduced the notion of 'structured space', i.e. a space which locally resembles various algebraic structures. In [2] and [3] we studied some cohomology theories related to these space. In this paper we continue in this direction: while in [2] we mainly focused on cohomologies arising from $f_s$ and $h$, and in [3] we considered cohomologies for generalisations of objects which involved structured spaces, here we deal with (co)homologies coming from the poset associated to a structured space via an equivalence relation defined at the end of Section 4 in [1]. More precisely, we will show that various (co)homologies for posets can also be applied (under some assumptions) to structured spaces.
\end{abstract}
\blfootnote{Author: \textbf{Manuel Norman}; email: manuel.norman02@gmail.com
}
\section{Introduction}
The main idea of this paper is to assign to a structured space some cohomologies which arise from the poset given by the relation $\sim$ in Section 4 of [1]. We have already found various cohomology theories related to these spaces, more precisely:\\
$\bullet$ In [2] we have developed two cohomology theories starting from $f_s$ and $h$ (these were suggested by the "similarity" of sheaves, vector bundles and the two maps above, as explained in the cited article)\\
$\bullet$ In [3] we have generalised various notions and we have developed the 'structured versions' of their cohomology theories\\
We now want to analyse structured spaces from another point of view: as we have seen in Section 4 of [1], the map $h$ is defined as follows (here $x \in X$, with $(X,f_s)$ structured space):
\begin{equation}\label{Eq:1.1}
x \xmapsto{h} \lbrace U_p \in \mathcal{U} : x \in U_p \rbrace \subseteq \mathcal{U}
\end{equation}
where $\mathcal{U}$ is, as usual (see [1]), the domain of the structure map $f_s$. The collection of all the families of subcollection of $\mathcal{U}$, excluding the empty collection, will be denoted, as in [1], by $\mathcal{L}$. In Section 4 of [1] we proved an important result on structured spaces (see Theorem 4.1), which asserts that, if $h: X \rightarrow \mathcal{L}$ satisfies some properties, then there is a semilattice $X/ \sim$ (or a lattice) corresponding to $X$ (we also proved the converse implication). Here, we will actually only need the previous part: $X / \sim$ is a poset (even if $h$ is not surjective or meet) under $\leq$, where we define:
\begin{equation}\label{Eq:1.2}
x \leq y \Leftrightarrow h(x) \subseteq h(y)
\end{equation}
and
\begin{equation}\label{Eq:1.3}
x \sim y \Leftrightarrow h(x) = h(y)
\end{equation}
(so $\leq$ is actually $<$). Indeed, it is not difficult to see that $\leq$ is a partial order on $X / \sim$. We will say that $X/ \sim$ is the 'poset corresponding/associated to the structured space $X$'. We will associate to $X$ some cohomology theories related to its corresponding poset, and the obtained cohomologies will be called '\textit{structured poset cohomologies}'. In each of the following sections, we will briefly recall some fundamental facts about the cohomologies for posets that we will need to consider; then, we will relate them with the structured space $X$.
\section{Simplicial (co)homology}
We refer to Chapter 4 in [14] for simplicial homology. For simplicial (co)homology of posets, we mainly refer to Section 1.5 and 1.6 of [9], and also to [17]. Briefly, given a poset $P$, we assign it its \textit{order complex} $\Delta(P)$, which is the abstract simplicial complex defined as follows:\\
$\bullet$ The vertices of $\Delta(P)$ are the points belonging to $P$\\
$\bullet$ The faces of $\Delta(P)$ are the totally ordered subsets (usually called 'chains') of $P$\\
For instance, $ \lbrace x, y \rbrace$ is a face of $\Delta(P)$, whichever are $x, y \in P$ such that $x \leq y$.\\
The simplicial (co)homology of a poset $P$ is the simplicial (co)homology of its order complex. This means that the chain complex we consider consists of the $\mathbb{K}$-modules freely generated by $n$-chains of $P$, that is, $C_n(P,\mathbb{K})$ ($\mathbb{K}$ is the ring of coefficients; we may take, for instance, $\mathbb{K}=\mathbb{Z}$). We recall that an $n$-chain is the following finite formal sum:
$$ \sum_{i=1}^{s} a_i \sigma_i $$
where $a_i \in \mathbb{K}$ and $\sigma_i$ is an oriented $n$-simplex. The boundary $\partial_n (\sigma)$ of an oriented $n$-simplex $\sigma=(t_0,t_1,...,t_n)$ is defined as follows:
$$ \partial_n (\sigma):= \sum_{i=0}^{n} (-1)^i (t_0,...,\widehat{t}_i, ..., t_n) $$
where as usual the cap $\widehat{t}_i$ means that we delete the vertex $t_i$ (so, $(t_0,...,\widehat{t}_i, ..., t_n)$ is the $i$-th face of $\sigma$). The map
$$ \partial_n : C_n(P,\mathbb{K}) \rightarrow C_{n-1}(P,\mathbb{K}) $$
is a homomorphism and $\partial_n \circ \partial_{n+1}=0$. So we have a chain complex for the order complex $\Delta(P)$, and thus we have associated a homology theory to the poset $P$. As usual, the homology groups are:
$$ H_n(P,\mathbb{K})=\ker \partial_n / \Ima \partial_{n+1} $$
Then, a simple and usual way to obtain the cohomology groups is the application of the functor $\Hom(\cdot, G)$, for some $\mathbb{K}$-module with identity $G$. See also [9] and [15]. Now we only have to apply this to our particular case $P=X/ \sim$.
\begin{theorem}\label{Thm:2.1}
Let $X$ be a structured space, and define the equivalence relation $\sim$ as in \eqref{Eq:1.3}. Then, $X / \sim$ is a partially ordered set (i.e. a poset) under $\leq$ (defined in \eqref{Eq:1.2}). We can thus consider the order complex of this poset, $\Delta(X / \sim)$, and evaluate its simplicial (co)homology. The resulting (co)homology theory is then called 'structured simplicial (co)homology' of $X$.
\end{theorem}
\begin{remark}\label{Rm:2.1}
\normalfont We note that, in general, if there exists a maximal element in the poset $P$, it is also useful to consider the Folkman complex of the poset, where such element is deleted (i.e. we consider $P \setminus 1$, where $1$ indicates the maximal element). Something similar can be done with the minimal element (usually denoted by $0$). This "cutting process" often simplifies the poset, and could allow simpler evaluations of the (co)homologies. We refer to Chapter 4 in [18] and to Chapter 4.5 in [20] for more details.\\
Moreover, we also notice that it is possible to obtain another (co)homology theory for structured spaces which is again related to posets and simplicial (co)homology. More precisely, apply Theorem 2.1 in [22] to a structured space $X$ (which is, by definition, a topological space): this (co)homology does not involve anymore $\leq$ and $\sim$ defined in \eqref{Eq:1.1} and \eqref{Eq:1.2}; however, as explained in [22], the two (different) equivalence relations are, in some sense, similar in their construction.
\end{remark}
\section{Cohomologies with coefficients in a presheaf}
Our main reference for this section is [11]. A poset $P$ can be seen as a category whose objects are the elements of $P$ and with the morphism $x \rightarrow y$ iff $x \leq y$. Consider a presheaf $\mathcal{F}:P^{op} \rightarrow \textbf{Ab}$ (we may sometimes drop the notation $P^{op}$, writing more briefly $P$: this should not cause confusion, since the domain of the presheaf certainly is the category $P^{op}$, not the topological space itself). The higher limits of $\mathcal{F}$ are defined as the right derived functors of the limit:
\begin{equation}\label{Eq:3.1}
\lim^i _{\longleftarrow P} := R^i \lim_{\longleftarrow P}
\end{equation}
We define the cohomology groups of $P$ with values in $\mathcal{F}$ as follows:
\begin{equation}\label{Eq:3.2}
H^p(P, \mathcal{F}):= \lim^p _{\longleftarrow P} \mathcal{F}
\end{equation}
This clearly specialises to the case $P=X/ \sim$.\\
We now consider cellular cohomology. Again, our main reference is [11]. Letting $x \leq z$ in the poset $P$, if for any $x \leq y \leq z$ we have either $y=x$ or $y=z$, then we write $x \prec z$. A poset $P$ is graded if there is a function $\rk : P \rightarrow \mathbb{Z}$ such that ($x < y$ means that $x \leq y$ and $x \neq y$):\\
(i) $x < y \Rightarrow \rk(x) < \rk(y)$\\
(ii) $x \prec y \Rightarrow \rk(y)=\rk(x) + 1$\\
Such a function is called a 'rank function'. If we suppose that some fixed $\rk$ is bounded above, say $r:=\max_P \rk(x)$, we can define the corank function $| \cdot | : P \rightarrow \mathbb{Z}$  by $|x|:= r - \rk(x)$. We now prove that, under some assumptions, $P=X/ \sim$ satisfies all the previous conditions.
\begin{proposition}\label{Prop:3.1}
Let $(X,f_s)$ be a structured space and suppose that $h$ is surjective onto $\mathcal{L}$. Moreover, assume that $\mathcal{U}$ is finite (i.e. it consists of a finite number of sets). Then, the function $\rk(x):=|h(x)|$ (the cardinality of the collection, defined as for sets; see also Section 4 of [2]) is a rank function, $r:=\max_{X / \sim} \rk(x) = 2^{|\mathcal{U}|}-1$, and thus $X/ \sim$ is a graded poset with a corank function.
\end{proposition}
\begin{proof}
The fundamental aspect to keep in mind here is that $h$ is surjective. We immediately notice that, because of the equivalence relation $\sim$ defined in \eqref{Eq:1.3}, $x \leq y$ is actually always a strict inequality, i.e. we always have $x < y$. This implies that, if $x < y$, then $h(x) \subsetneq h(y)$, and thus $\rk(x) < \rk(y)$. Furthermore, since $h$ is surjective, we can say that:\\
1) If $x \leq y \leq z$ implies either $y=x$ or $y=z$, then either $h(x) = h(y) \subseteq h(z)$ or $h(x) \subseteq h(y) = h(z)$.\\
2) By surjectivity, we know that the cardinality of $h$ does not "jump" any value, i.e. if there are $x_1, x_2$ such that $|h(x_1)|=n$ and $|h(x_2)|=t$ (say, with $n \leq t$; the other case can be treated analogously) for some natural numbers $n,k$ (which are certainly $\neq 0$), then there must exist some points $y_1, y_2, ..., y_k$ such that $|h(y_1)|=n+1$, $|h(y_2)|=n+2$, ..., $|h(y_k)|=t-1$.\\
3) By the two observations above, we can conclude that $x \prec y \Rightarrow \rk(x)=\rk(y)+1$. Indeed, suppose that this does not hold. It is clear that we then have $\rk(y)>\rk(x)+1$ (by definition of $\leq$, and by definition of cardinality). But this implies that there exists at least one point $a$ such that $\rk(x) < \rk(a) < \rk(y)$. Consequently, it is not anymore true that whenever $x \leq z \leq y$ implies either $x=z$ or $z=y$, because $z=a$ is a counterexample. This is absurd, and thus the statement above holds.\\
This rank function is bounded above, because $\mathcal{U}$ is finite and we know that (as noted in Section 4 of [1]) $\mathcal{L}$ is the "power collection" without the empty collection, which clearly implies that the maximum cardinality is \footnote{Since $h$ is surjective, there is at least one $x$ such that $h(x)= \mathcal{U}$.} $2^{|\mathcal{U}|}-1$. We therefore conclude that $X/ \sim$ is graded and it has a corank function.
\end{proof}
Thanks to this result, we know that, under its hypothesis, we can apply the theory in Section 2 of [11]. We first need to recall some other facts and notations, for which we always refer to [11]. First of all, the filtrations of $P$ by the corank defined on it are:
\begin{equation}\label{Eq:3.3}
P^k:=\lbrace \, x \in P : |x| \leq k \, \rbrace
\end{equation}
with $k \in \mathbb{N}_0$. Consequently, we have $P^0 \subset P^1 \subset P^2 \subset ...$, and a presheaf $\mathcal{F}$ on $P$ can be defined on each $P^k$ thanks to the inclusions $P^k \hookrightarrow P$. Now, it is possible to compute the higher limits defined by \eqref{Eq:3.1} via some groups, which will be indicated by $HS^p$. A description of this computational method can be found, for instance, in Section 1.1 of [11]. It can be proved that ($\cong$ means 'isomorphic'):
\begin{equation}\label{Eq:3.4}
\lim^p _{\longleftarrow P} \mathcal{F} \cong HS^p(P,\mathcal{F})
\end{equation}
for all $p$. We can also define (see Section 1.4 in [11]) a relative cohomology $HS^p(P_1, P_2, \mathcal{F})$. This allows us to finally define the cellular cochain complex associated to a poset $P$. The groups in the cochain are:
\begin{equation}\label{Eq:3.5}
C^n(P, \mathcal{F}):=HS^n(P^n, P^{n-1}, \mathcal{F})
\end{equation}
while the coboundary maps $\delta^n$ are defined via Lemma 4 in [11]; see also equation (9) in the cited paper. We can thus conclude with the follwing important:
\begin{theorem}\label{Thm:3.1}
Let $(X,f_s)$ be a structured space and consider a presheaf $\mathcal{F}$ on the poset $X/ \sim$. Then, the higher limits defined in \eqref{Eq:3.1} allows us to define the cohomology of  $X/ \sim$ with values in the presheaf $\mathcal{F}$ via \eqref{Eq:3.2}. Moreover, if $h:X \rightarrow \mathcal{L}$ is surjective and the domain $\mathcal{U}$ of $f_s$ is finite (i.e. it contains a finite number of sets), $X/ \sim$ is a graded poset with corank, with the rank defined by $\rk(x):=|h(x)|$ (the cardinality of the collection). It is then possible to define also the cellular cohomology of $X / \sim$, see \eqref{Eq:3.5} and the comments below it. All these cohomologies are called 'structured cohomologies with values in a presheaf' of $X$.
\end{theorem}
\begin{remark}\label{Rm:3.1}
\normalfont The surjectivity of $h$ implies the existence of at least one element $x \in P$ such that $h(x)= \mathcal{U}$. This implies that there exists at least one maximal element, which is unique when $\bigcap_{U_p \in \mathcal{U}} U_p$ has one and only one element (see also the next Section). Thus, it can be useful to consider also the Folkman complex (see Remark \ref{Rm:2.1}). See also Section 4.2 in [11] and [21].
\end{remark}
\section{Structured coloured (co)homology}
In this section we consider a (co)homology for coloured posets. Our main reference here is [10]. By Definition 1 in the cited paper, we know that a coloured poset is a couple $(P, \mathcal{F})$, where $P$ is a poset with a unique maximal element (i.e. there exists an element $\widetilde{x}$ which is not smaller than any other element in $P$; moreover, this element is required to be unique: the possibility of having more than one maximal element is due to the partial order), and $\mathcal{F}:P \rightarrow \textbf{Mod}_R$ (the category of $R$-modules, for some unital ring $R$) is a covariant functor. $\mathcal{F}$ is called the 'colouring'. We now show a particular case where we can consider a colouring in relation with structured spaces.
\begin{proposition}\label{Prop:4.1}
Let $(X,f_s)$ be a structured space and suppose that 
$$\bigcap_{U_p \in \mathcal{U}} U_p = \lbrace \, \widetilde{x} \, \rbrace$$
(precisely one element). Then, $\widetilde{x}$ is the unique maximal element of the poset $X/ \sim$ (under $\leq$).
\end{proposition}
\begin{proof}
Since 
$$\bigcap_{U_p \in \mathcal{U}} U_p = \lbrace \, \widetilde{x} \, \rbrace$$
we know that $h(\widetilde{x})= \mathcal{U}$ (because it intersects all the fixed neighborhoods). Moreover, it is clear that there is no other element in $X/ \sim$ such that this happens. The fact that $h(y) \subsetneq h(\widetilde{x})$ $\forall y \in X/ \sim$ then implies that $\widetilde{x}$ is a maximal element, and the uniqueness follows from the observation above.
\end{proof}
Of course, there are also other cases for which a unique maximal element exists. However, here we will mainly consider the above one. We can thus define:
\begin{definition}\label{Def:4.1}
Given a structured space $(X,f_s)$, the couple $(X/ \sim, \mathcal{F})$, where $\mathcal{F}: X/ \sim \,  \rightarrow \textbf{Mod}_R$ is a covariant functor, and $X/ \sim$ has a unique maximal element, is a structured coloured poset. A 'maximal structured coloured poset' is a couple $(X/ \sim, \mathcal{F})$ where
$$\bigcap_{U_p \in \mathcal{U}} U_p = \lbrace \, \widetilde{x} \, \rbrace$$
which is a particular case of structured coloured poset by Proposition \ref{Prop:4.1}.
\end{definition}
The name 'maximal' comes from the fact that $|h(\widetilde{x})|=|\mathcal{U}|$ has the maximum cardinality among all the collections $h(y)$ for $y \in X$. The unique maximal element $\widetilde{x}$ of a structured coloured poset will be indicated, as usual in this context, by $1$.\\
Following [10], if $(P, \mathcal{F})$ is a coloured poset, we can define the chain complex ($n >0$):
\begin{equation}\label{Eq:4.1}
S_n(P, \mathcal{F}):= \bigoplus_{x_1 x_2 \cdot \cdot \cdot x_n, \, x_i \in P \setminus 1} \mathcal{F}(x_1)
\end{equation} 
and $S_0(P, \mathcal{F}):=\mathcal{F}(1)$, where $\textbf{x}=x_1 x_2 \cdot \cdot \cdot x_n$ means $x_1 \leq x_2 \leq ... \leq x_n$. For $n>0$, an element of the chain complex can be written as:
$$ \sum_{\textbf{x}} \lambda \cdot \textbf{x}$$
where the sum is over the sequences \textbf{x} of lenght $n$, and $\lambda \in \mathcal{F}(x_1)$. The differentials $d_n: S_n(P, \mathcal{F}) \rightarrow S_{n-1}(P, \mathcal{F})$ are defined as follows ($n>1$):
\begin{equation}\label{Eq:4.2}
d_n(\lambda x_1 x_2 \cdot \cdot \cdot x_n):= \mathcal{F}^{x_2} _{x_1}(\lambda) x_2 \cdot \cdot \cdot x_n - \sum_{i=2}^{n} (-1)^i \lambda x_1 \cdot \cdot \cdot \widehat{x}_i \cdot \cdot \cdot x_n
\end{equation}
and $d_1(\lambda x):=\mathcal{F}^1 _x (\lambda)$. Something similar can be done using strict inequalities $<$ instead of $\leq$. The resulting complex is denoted by $C_n(P, \mathcal{F})$. In [10] various results involving also this complex are obtained. Here, our main interest was to show that this construction also applies to some structured spaces, so we will not go deeper. We can define a cohomology via the usual application of the functor $\Hom(\cdot, R)$ (where $R$ is the same as before). We then have:
\begin{theorem}\label{Thm:4.1}
Let $(X,f_s)$ be a structured space. Suppose that $X / \sim$ has a unique maximal element, indicated by $1$ (see, for instance, maximal structured coloured posets) and let $\mathcal{F}: X/ \sim \, \rightarrow \textbf{Mod}_R$ be a covariant functor ($R$ is some unital ring). Then, $(X,\mathcal{F})$ is a structured coloured poset to which we can associate a coloured homology theory (see \eqref{Eq:4.1} and \eqref{Eq:4.2} with $P=X/ \sim$) and a cohomology theory obtained by applying to the homology, as usual, the functor $\Hom( \cdot, R)$.
\end{theorem}
\section{Structured stratified spaces}
We conclude this paper with two (co)homologies for a particular kind of poset-stratified spaces, called 'structured stratified spaces'. Recall that (see, for instance, [25-30]) a poset-stratified space is a structure consisting of a topological space $X$, a poset $P$ with the Alexandroff topology and a continuous surjection $s:X \rightarrow P$. Our previous discussions suggest an interesting case from the point of view of the theory of structured spaces: to consider a structured space $X$ and the corresponding poset $X/ \sim$ endowed with the Alexandroff topology.
\begin{definition}\label{Def:5.1}
A structured stratified space consists of a structured space $X$ endowed with the topology in Proposition 1.1 in [1] (more precisely, the one constructed in Example 1.1 of such paper), its corresponding poset $X/ \sim$ endowed with the Alexandroff topology, and a continuous surjection $s:X \rightarrow X/ \sim$.
\end{definition}
We can apply the theory of stratified spaces to structured spaces thanks to the above definition; this gives another tool to study these spaces. We briefly sketch some possible directions to do this, even though we will soon return to the main topic of this paper. Following [25], it is clear (by Definition 4.2.1) that structured stratified spaces are indeed a particular case of stratified spaces. By Construction 4.2.3, we can endow the structured space $X$ with a preorder given by the considered map $s$ as follows:
$$ x \leq_s y \Leftrightarrow sx \leq sy \, \text{in} \, X/ \sim $$
Consequently, by 5.1.7 we can define a precicurlation $\leq_s|_{\bullet}$ on the preordered space $X$ via restrictions. This implies that $(X, \leq_s|_{\bullet})$ is a prestream (see Definition 5.1.1). We will call it the 'prestream associated to the structured stratified space'. It is also possible to consider other precirculations for structured stratified spaces, which could turn out to make the prestream into a stream. There are various possible definitions of this concept: Definition 5.1.14 involves Haucourt streams (see also [27]); another definition, which can be found in Remark 5.1.19, is given by Krishnan in [28]. We can also define the 'd-space associated to a structured stratified space', which is the d-space $(X,d^{\leq_s|_{\bullet}}X)$ defined as in 5.1.11 (see 5.1.10  for the definition of d-space).\\
We now return to our main topic: we first give a "standard" example of structured stratified space. In fact, the following example will be called 'standard structured stratified space associated to a structured space', and we will always refer to it if not specified otherwise.
\begin{proposition}\label{Prop:5.1}
Let $X$ be a structured space endowed with the topology in Proposition 1.1 (more precisely, Example 1.1) of [1], and let $X/ \sim$ be its corresponding poset endowed with the Alexandroff topology. Moreover, suppose that $\sup_{x \in X} |h(x)| < +\infty$ \footnote{The cardinality of a collection of sets is defined similarly to sets; see, for instance, Section 4 in [2].}. Then, $X$, $X/ \sim$ and the map $s:X \rightarrow X/ \sim$ given by $s(x):=[x]$ form a structured stratified space. 
\end{proposition}
\begin{proof}
We only need to verify that $s$ is continuous and surjective. Surjectivity is clear, since all the elements $[x]$ in $X/ \sim$ are reached at least by $x \in X$. In order to prove continuity, recall that (see, for instance, Definition 4.1.13 in [25]), by definition of Alexandroff topology, if $U \subseteq X/ \sim$ is open, then whenever $p \in U$ and $t \in X/ \sim$, with $p \leq t$, we have $t \in U$. Now consider some open $U$ in $X/ \sim$; we have $s^{-1}(U)=\lbrace x \in X: s(x) \in U \rbrace$. By definition of the equivalence classes in $X/ \sim$, we know that all the elements $x \in X$ such that $s(x)=[x]$ satisfy the following property:
$$h(x)=\lbrace U_p \rbrace_{\text{for some} \, U_p \in \mathcal{U}}$$
for some fixed $U_p$'s in $\mathcal{U}$ (they are precisely the same for all these $x$, and no other $y \in X$ is such that $h(y)$ is equal to the above collection). This means that all the points $x$ with the same equivalence class form a set $A=\bigcap_p U_p \setminus (\bigcup_t U_t)$, where the intersection is over the above $U_p$'s, while the $U_t$'s are sets not belonging to $h(x)$ but which intersect with the set $\bigcap_p U_p$. However, since an open set $U \subseteq X/ \sim$ satisfies the property previously noticed (thanks to Alexandroff topology), we know that $s^{-1}(U)$ will be precisely the union of some intersections as the above one, because the gaps due to the differences of sets will be filled by the points $y$ such that $x \leq y$ (which belong to $U$). Consequently, $s^{-1}(U)$ is a union of open sets, since (under the assumption that the number of open sets in all the considered interesections is finite) the intersections are open. The Proposition follows.
\end{proof}
Now, following [26], to each structured stratified space $X$ we can associate a simplicial set, called 'stratified singular simplicial set of $X$'. This set is denoted by $SS(X)$; see Definition 7.1.0.3. We can evaluate the homology of this simplicial set and associate it to the structured stratified space (also recall that this homology is isomorphic to the singular homology of the geometric realisation of the simplicial set). More precisely, we can associate to a simplicial set $Y$ the chain:
$$ C_n= \mathbb{Z}[Y_n] $$
(the free abelian group on $Y_n$) and the boundary maps:
$$ \sum_i (-1)^i d_i : C_n \rightarrow C_{n-1} $$
(the alternating sum of the face maps). Another possible homology can be obtained choosing some simplicial set $A$ and considering the homology of the simplicial set $\emph{\textbf{sset}}(A, SS(X))$ ($X$ structured stratified space); see also Lemma 7.3.0.1. We can obtain cohomologies in the usual way. We thus conclude, summing up what we have said up to now:
\begin{theorem}[(Co)homology for structured stratified spaces]\label{Thm:5.1}
Let $(X, X \xrightarrow{s} X/ \sim)$ be a structured stratified space. We can define at least two (co)homologies for such a space: one is given by the (co)homology of the simplicial set $SS(X)$, while another one is given by the (co)homology of the simplicial set $\emph{\textbf{sset}}(A, SS(X))$ (after choosing some simplicial set $A$).
\end{theorem}
Everything clearly specialises to the particular case of the standard structured stratified space:
\begin{theorem}[Structured stratified (co)homology]\label{Thm:5.2}
Let $X$ be a structured space and consider the standard structured stratified space associated to it. Then, a structured stratified (co)homology of $X$ is defined to be one of the possible (co)homologies of $(X, X \xrightarrow{s} X/ \sim)$.
\end{theorem}
\section{Conclusion}
In this paper we have developed other (co)homology theories for structured spaces arising, this time, from their corresponding poset. The methods used can be found in the cited references; the main aim of this paper is to notice that these theories can also be applied (under some hypothesis) to the case of structured spaces using the poset associated via $\sim$ and $h$.\\
\\
\begin{large}
\textbf{References}
\end{large}
\\
$[1]$ Norman, M. (2020). On structured spaces and their properties. Preprint (arXiv:2003.09240)\\
$[2]$ Norman, M. (2020). Two cohomology theories for structured spaces. Preprint (arXiv:2004.11152)\\
$[3]$ Norman, M. (2020). Structured spaces: categories, sheaves, bundles, schemes and cohomology theories. In preparation\\
$[4]$ Everitt, B.; Turner, P. (2012). Bundles of coloured posets and a Leray-Serre spectral sequence for Khovanov homology. Transactions of the American Mathematical Society, 364(6), 3137-3158\\
$[5]$ Brun, M.; Bruns, W.; R\"omer, T. (2007). Cohomology of partially ordered sets and local cohomology of section rings. Advances in Mathematics, 208(1): 210-235\\
$[6]$ D\'iaz, A. (2007). A method for integral cohomology of posets. Preprint (arXiv:0706.2118)\\
$[7]$ Cianci, N.; Ottina, M. (2017). A new spectral sequence for homology of posets. Topology and its Applications, Volume 217, 1-19\\
$[8]$ Ramos, A. D. (2006). Homological algebra on graded posets. PhD thesis. University of Malaga\\
$[9]$ Wachs, M. L. (2007). Poset topology: tools and applications. Geometric Combinatorics, IAS/Park City Math. Series 13, (Miller, Reiner, and Sturmfels, eds.), American Mathematical Society, Providence, RI,497-615\\
$[10]$ Everitt, B.; Turner, P. (2009). Homology of coloured posets: A generalisation of Khovanov's cube construction. Journal of Algebra, Vol. 322, No. 2, 429-448\\
$[11]$ Everitt, B.; Turner, P. (2015). Cellular cohomology of posets with local coefficients. Journal of Algebra, Vol. 439, 134-158\\
$[12]$ Cianci, N.; Ottina, M. (2019). Homology of posets with functor coefficients and its relation to Khovanov homology of knots. Preprint (arXiv:1907.03974)\\
$[13]$ Quillen, D. (1978). Homotopy properties of the poset of nontrivial p-subgroups of a group. Advances in Mathematics 28, 2, 101-128\\
$[14]$ Deo, S. (2018). Algebraic Topology. Texts and Readings in Mathematics, vol 27. Springer, Singapore\\
$[15]$ Gallier, J.; Quaintance, J. (2019). A Gentle Introduction to Homology, Cohomology, and
Sheaf Cohomology, University of Pennsylvania\\
$[16]$ Gabriel, P.; Zisman, M. (1967). Calculus of fractions and homotopy theory. Ergebnisse der Mathematik und ihrer Grenzgebiete, Band 35. Springer-Verlag New York, Inc., New York\\
$[17]$ Bj\"orner, A. (1995). Topological methods, Handbook of combinatorics, Vol. 1, 2, Elsevier, Amsterdam, 1819-1872\\
$[18]$ Orlik, P. (1989). Introduction to Arrangements. CBMS Regional Conference Series in Mathematics, vol. 72\\
$[19]$ Folkman, J. (1966) The homology groups of a lattice. J. Math. Mech. 15, 631-636\\
$[20]$ Orlik, P.; Terao, H. (1992). Arrangements of Hyperplanes. Grundlehren der mathematischen Wissenschaften, vol 300. Springer-Verlag Berlin Heidelberg\\
$[21]$ Boussicault, A. (2013). Operations on partially ordered sets and rational identities of type A. Discret. Math. Theor. Comput. Sci., Vol 15, No 2, 13-32\\
$[22]$ Norman, M. (2020). A (co)homology theory for some preordered topological spaces. Preprint (arXiv:2004.10689)\\
$[23]$ Haine, P. J. (2018). On the homotopy theory of stratified spaces. Preprint (arXiv: 1811.01119)\\
$[24]$ Chandler, A.; Sazdanovic, R. (2019). A Broken Circuit Model for Chromatic Homology Theories. Preprint (arXiv:1911.13226)\\
$[25]$ Nicotra, S. (2020). A convenient category of locally stratified spaces. PhD thesis, University of Liverpool\\
$[26]$ Nand-Lal, S. J. (2019). A simplicial approach to stratified homotopy theory. PhD thesis, University of Liverpool\\
$[27]$ Haucourt, E. (2012). Streams, d-spaces and their fundamental categories. Electronic Notes in Theoretical Computer Science 283, 111-151\\
$[28]$ Krishnan, S. (2009). A convenient category of locally preordered spaces. Applied Categorical Structures 17.5, 445-466\\
$[29]$ Yokura, S. (2017). Decompostion spaces and poset-stratified spaces. Preprint (arXiv:1912.00339)\\
$[30]$ Goubault-Larrecq, J. (2014). Exponentiable streams and prestreams. Applied Categorical Structures 22.3, 515-549

\end{document}